\documentclass{ws-ijac}
\usepackage{lmodern}
\usepackage{graphicx}
\usepackage{rflatexmacs}
\usepackage{wjdlatexmacs}
\usepackage{url,amssymb,enumerate,tikz,scalefnt}
\usepackage[normalem]{ulem} % for \sout (strikeout)   wjd: could remove this in final draft
\usepackage{algorithm2e}

\usepackage[mathcal]{euscript}
\newcommand{\mysetminus}{\ensuremath{-}}
\newtheorem{prob}{Problem}

\begin{document}

\markboth{W.~DeMeo, R.~Freese, M.~Valeriote}
{Tests for Difference Terms}

%%%%%%%%%%%%%%%%%%%%% Publisher's Area please ignore %%%%%%%%%%%%%%%
%
\catchline{}{}{}{}{}
%
%%%%%%%%%%%%%%%%%%%%%%%%%%%%%%%%%%%%%%%%%%%%%%%%%%%%%%%%%%%%%%%%%%%%

\title{POLYNOMIAL-TIME TESTS FOR DIFFERENCE TERMS IN IDEMPOTENT VARIETIES}

% \author[W.~DeMeo]{William DeMeo}
% \email{williamdemeo@gmail.com}
% \urladdr{http://williamdemeo.github.io}
% \address{University of Colorado\\Mathematics Dept\\Boulder 80309\\USA}
\author{WILLIAM DEMEO}
\address{Mathematics Department, University of Colorado\\Boulder, Colorado 80309, USA\\
\email{williamdemeo@gmail.com}}

% \author[R.~Freese]{Ralph Freese}
% \email{ralph@math.hawaii.edu}
% \urladdr{http://www.math.hawaii.edu/~ralph/}
% \address{University of Hawaii\\Mathematics Dept\\Honolulu 96822\\USA}
\author{RALPH FREESE}
\address{Mathematics Department, University of Hawaii\\Honolulu, Hawaii 96822, USA\\
\email{ralph@math.hawaii.edu}}

% \author[M.~Valeriote]{Matthew Valeriote}
% \email{matt@math.mcmaster.ca}
% \urladdr{http://ms.mcmaster.ca/~matt/}
% \address{McMaster University\\Mathematics Dept\\Hamilton L8S 4K1\\
% CAN}
\author{MATTHEW VALERIOTE}
\address{Mathematics Department, McMaster University\\Hamilton, Ontario L8S 4K1, CAN\\
\email{matt@math.mcmaster.ca}}

% \thanks{The first and second authors were supported by the National
% Science Foundation under Grant Numbers 1500218 and 1500235; the third author was supported by a grant from the Natural Sciences and Engineering Research Council of Canada.}

% \date{\today}
% \begin{document}
\maketitle

\begin{history}
\received{(Day Month Year)}
\accepted{(Day Month Year)}
\comby{[editor]}
\end{history}

\begin{abstract}
We consider the following practical question: given a finite
algebra $\alg{A}$ in a
finite language, can we efficiently decide whether the variety
generated by $\alg{A}$
has a difference term?  We answer this question (positively) in the
idempotent case and then describe algorithms for constructing difference
term operations.
\end{abstract}

\section{Introduction}
\label{sec:introduction}

A \defn{difference term} for a variety $\sV$ is a ternary term $d$ in the
language of $\sV$ that satisfies the following:
if $\alg{A} = \<A, \dots \> \in \sV$, then for all $a, b \in A$ we have
\begin{equation}
\label{eq:3}
d^{\alg{A}}(a,a,b) = b \quad \text{ and } \quad
d^{\alg{A}}(a,b,b) \mathrel{\comm \theta \theta} a,
\end{equation}
where $\theta$ is any congruence %% of $\alg{A}$
containing $(a,b)$
and $[\cdot, \cdot]$ denotes the \defn{commutator}.
%% (see Section~\ref{sec:defin-notat}).
When the relations in (\ref{eq:3}) hold for a single algebra $\alg{A}$ and term $d$ we call $d^{\alg{A}}$
a \defn{difference term operation} for $\alg{A}$.

Difference terms are studied extensively in the general algebra literature.
(See, for example, \cite{MR1358491,MR1663558,MR3076179,KSW,MR3449235}.)
There are many reasons to study difference terms, but
one obvious reason is because if we know that a variety
has a difference term, this fact allows us to deduce some useful
properties of the algebras inhabiting that variety.
In particular, varieties with a difference term have a commutator that behaves nicely.  

The class of varieties that have a difference term is fairly broad and includes those varieties that are congruence modular or congruence meet-semidistributive. Since the commutator of two congruences of an algebra in a congruence meet-semidistributive variety is just their intersection~\cite{MR1663558}, it follows that the term $d(x,y,z) := z$ is a difference term for such varieties.  A special type of difference term $d(x,y,z)$ is one that satisfies the equations $d(x,x,y) = y$ and $d(x,y,y) = x$.  Such terms are called \emph{\malcev\ terms}. So if $\alg A$ lies in a variety that has a difference term $d(x,y,z)$ and if $\alg A$ is \emph{abelian} (i.e., $[1_A, 1_A] = 0_A$), then $d$ will be a \malcev\ term for $\alg A$.

%Any variety that has a \malcev\ term or for which the commutator operation satisfies $[\alpha, \beta] = \alpha \meet \beta$ has a difference term.
%%Roughly speaking, having a difference term is slightly stronger than having
%%a Taylor term and slightly weaker than having a \malcev term.
%(Note that if
%$\alg{A}$ is an \defn{abelian} algebra, which means
%that $[1_A, 1_A] = 0_A$, then, by
%the monotonicity of the commutator,
%$[\theta, \theta] = 0_A$ for all $\theta \in \Con \alg{A}$,
%in which case % $\textbf{A}$
%(\ref{eq:3}) says that $d^{\alg{A}}$ is a \malcev term operation.)

Difference terms also play a role in recent work of Keith Kearnes,
Agnes Szendrei, and Ross Willard.
In~\cite{MR3449235} these authors give a positive answer
J\'onsson's famous question---whether a variety of finite
residual bound must be finitely
axiomatizable---for the
special case in which the variety has a difference
term.\footnote{To say a variety has \emph{finite residual bound} is to say
  there is a finite bound on the size of the subdirectly irreducible
  members of the variety.}

Computers have become invaluable as a research tool and have helped to
broaden and deepen our understanding of algebraic structures and the
varieties they inhabit.  This is largely due to the efforts
of researchers who, over the last three decades, have found ingenious
ways to coax computers into solving challenging abstract algebraic
decision problems, and to do so very quickly.
To give a couple of examples related to our own work,
it is proved in~\cite{MR3239624} (respectively,~\cite{Freese:2009})
that deciding whether a finite idempotent algebra generates a variety that is congruence-$n$-permutable
(respectively, congruence-modular) is \emph{tractable}.\footnote{To
  say that the decision problem is \emph{tractable} is to say
  that there exists an algorithm for solvng the problem that ``scales
  well'' with respect to increasing input size, by which we mean that
  the number of operations required to reach a correct decision is
  bounded by a polynomial function of the input size.}
The present paper continues this effort by presenting an efficient
algorithm for deciding whether a finitely generated idempotent variety has a
difference term.

The question that motivated us to begin this project, and
whose solution is the main subject of this paper, is the following:
\begin{prob}
  \label{prob:1}
  Is there a polynomial-time algorithm to decide for a finite,
  idempotent algebra $\alg{A}$ if $\bbV(\alg{A})$ has a difference term?
\end{prob}

We note that for arbitrary finite algebras $\alg{A}$, the problem of deciding if $\bbV(\alg{A})$ has a difference term is an EXP-time complete problem.  This follows from Theorem 9.2 of \cite{Freese:2009}.

The remainder of this introduction uses the language of \emph{tame congruence theory} (abbreviated by \tct).  Many of the terms we use are defined and explained
in the next section.  For others, see~\cite{HM:1988}.

Our solution to Problem~\ref{prob:1} exploits the connection
between difference terms and \tct that was established
by Keith Kearnes in~\cite{MR1358491}.
%The following theorem makes this connection precise.

\begin{theorem}[{\protect\cite[Theorem 1.1]{MR1358491}}]
\label{thm:KearnesThm}
The variety $\sV = \bbV(\alg A)$ generated by a
finite algebra $\alg A$ has a difference term  if and only if
$\sV$ omits \tct-type \utyp and, for all finite algebras
$\alg B \in \sV$,
the minimal sets of every type~\atyp\ prime interval in
$\op{Con}(\alg B)$ have empty tails.
\end{theorem}

%%% NOTATION FOR TCT TYPES
% \newcommand{\otyp}{\textbf{0}}
% \newcommand{\utyp}{\textbf{1}}
% \newcommand{\atyp}{\textbf{2}}
% \newcommand{\btyp}{\textbf{3}}
% \newcommand{\ltyp}{\textbf{4}}
% \newcommand{\styp}{\textbf{5}}
% \newcommand{\ityp}{\textbf{i}}
% \newcommand{\jtyp}{\textbf{j}}

It follows from an observation of Bulatov that the problem of deciding if a finite idempotent algebra generates a variety that omits \tct-type \utyp is tractable (see Proposition 3.1 of \cite{MR2504025}).
In~\cite{Freese:2009}, the second and third authors of this paper solve an
analogous problem by giving a positive answer to the following:
\begin{prob}
  \label{prob:2}
  Is there a polynomial-time algorithm to decide for a finite,
  idempotent algebra $\alg{A}$ if $\bbV(\alg{A})$ is congruence modular?
\end{prob}

Congruence modularity is characterized by omitting tails and
\tct-types \utyp and \styp.
Omitting \utyp's and \styp's can be efficiently decided using Corollary~\ref{cor:2.2}. %by the subtype theorem.
It is also proved in~\cite{Freese:2009} that if there is a nonempty tail in $\bbV(\alg{A})$, then there is a nonempty tail in a ``small'' algebra. This is then used to show that the problem of deciding if a finite idempotent algebra generates a congruence modular variety is tractable. More precisely, suppose $\alg{A}$ is a finite idempotent algebra, and suppose
$\bbV(\alg{A})$ has nonempty tails but lacks \utyp's and \styp's. Then a nonempty tail must occur in a 3-generated subalgebra of $\alg{A}^2$. A key fact used in the proof of this is that a finite algebra in a locally finite variety that omits \utyp's and \styp's has a congruence lattice that---modulo the {\it solvability congruence} ---is (join) semidistributive.
%The authors use this to prove that congruence modularity is polynomial-time decidable.

%However, proving lack of tails uses the fact that a variety omitting
%\utyp's and \styp's has a congruence lattice that---modulo
%the {\it solvability congruence} (defined below)---is (join) semidistributive.
%Now, restricting to just testing whether $\bbV(\alg{A})$ omits
%type-\atyp\ tails is not a problem. So, for example, there is a
%polynomial-time algorithm for testing if
%$\bbV(\alg{A})$ omits \utyp's, \styp's, and type-\atyp\ tails.

%Here is a related problem.
%\wjd{deleted the related problem; it no longer seems relevant.}
%

%% Hobby and McKenzie give some info about the types in a $D_2$
%% embedded in $\Con (\alg{A})$. (See~\cite[Lemma 6.3]{HM:1988}).
%% Exercise 7 of that section considers 4-element
%% algebras whose congruence lattice is the concrete embedding of $D_2$
%% in $\Eq(4)$; the one with coatoms $01|23$, $02|13$, and $0|123$, and with atoms
%% $0|1|23$ and $0|2|13$. By~\cite[Lemma 6.3]{HM:1988}, the middle-top
%% interval must be type 5 (assuming a Taylor term). But all the others can be 5's,
%% or all the others can be 4's, or all the others can be 3's.
%% One might attempt to find an example where they are all 2's, but that not possible
%% since otherwise $0|123$ would be a solvable congruence,
%% which would imply the two atoms would permute.

%% \draftbreak

\section{Background, definitions, and notation}
\label{sec:defin-notat}
Our starting point is the set of lemmas at the beginning of Section 3 of~\cite{Freese:2009}.
We first review some of the basic \ac{tct}
that comes up in the proofs in that paper. (In fact, most of this section
is based on~\cite[Section~2]{Freese:2009}.)

The seminal reference for \tct is the book by Hobby and McKenzie
\cite{HM:1988}, according to which,
for each covering $\alpha \prec \beta$ in the congruence lattice of a finite
algebra $\alg{A}$, the local behavior of the $\beta$-classes is captured by the
so-called $(\alpha, \beta)$-traces~\cite[Def.~2.15]{HM:1988}.
Modulo $\alpha$, the induced structure on the traces is limited to one
of five possible types:

\begin{enumerate}[{\bf 1}]
\item  (unary type) an algebra whose basic operations are permutations;
\item  (affine type) a one-dimensional vector space over some finite field;
\item  (boolean type) a 2-element boolean algebra;
\item  (lattice type) a 2-element lattice;
\item  (semilattice type) a 2-element semilattice.
\end{enumerate}

Thus to each covering $\alpha \prec \beta$
corresponds a ``\tct type,'' denoted by $\typ(\alpha, \beta)$,
belonging to the set
$\{\mathbf{1},\mathbf{2},\mathbf{3},\mathbf{4},\mathbf{5}\}$
(see~\cite[Def.~5.1]{HM:1988}).
The set of all \tct types that are realized by covering pairs of congruences of a
finite algebra $\alg{A}$ is denoted by $\typ\{\alg{A}\}$
and called the \emph{typeset} of $\alg{A}$.
If $\sK$ is a class of algebras, then $\typ\{\sK\}$ denotes the union of the typesets of all finite algebras in $\sK$.
\tct types are ordered according to the following ``lattice of types'':

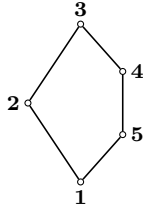
\newcommand{\dotsize}{0.8pt}
%% To create nodes of lattices in a uniform and consistent way, we define
\tikzstyle{lat} = [circle,draw,inner sep=\dotsize]
% To scale all diagrams uniformly, change this setting:
\begin{center}
\newcommand{\figscale}{.7}
\begin{tikzpicture}[scale=\figscale]
  \scalefont{.8}
  \node[lat] (1) at (0,0) {};
  \node[lat] (2) at (-1,1.5) {};
  \node[lat] (3) at (0,3) {};
  \node[lat] (4) at (.8,2.1) {};
  \node[lat] (5) at (.8,.9) {};
  \draw (1) node [below] {$\mathbf{1}$};
  \draw (2) node [left] {$\mathbf{2}$};
  \draw (3) node [above] {$\mathbf{3}$};
  \draw (4) node [right] {$\mathbf{4}$};
  \draw (5) node [right] {$\mathbf{5}$};
  \draw[semithick]
  (1) -- (2) -- (3) -- (4) -- (5) -- (1);
\end{tikzpicture}
\end{center}
For $\alg{A}$ a finite idempotent algebra, whether or not $\bbV(\alg{A})$ omits one of the order ideals of the lattice of types can be
determined locally.  This is spelled out for us in the next proposition.
(A \defn{strictly simple} algebra is a simple
algebra with no non-trivial subalgebras.)
%% ; i.e.~no proper subalgebras with
%% more than one element.)

\begin{proposition}[{\protect \cite[Proposition~2.1]{Freese:2009}}]
  \label{prop:2.1}
If $\alg A$ is a finite idempotent algebra and
$\mathbf{i} \in \typ(\bbV(\alg{A}))$ then there
is a finite strictly simple algebra $\bS$ of
type~$\mathbf{j}$ for
some $\mathbf{j} \leq \mathbf{i}$ in $\sansH \sansS (\alg{A})$.
The possible cases are
% \begin{enumerate} %[(1)]
\begin{itemize}
\item[$\bullet\  \mathbf{j} = \mathbf{1}$] $\;\Rightarrow \;\bS$ is term equivalent to a 2-element set
\item[$\bullet\   \mathbf{j} = \mathbf{2}$] $\;\Rightarrow \;\bS$ is term equivalent to the idempotent reduct of a module
\item[$\bullet\   \mathbf{j} = \mathbf{3}$] $\;\Rightarrow \;\bS$ is functionally complete
\item[$\bullet\   \mathbf{j} = \mathbf{4}$] $\;\Rightarrow \;\bS$ is polynomially equivalent to a 2-element lattice
\item[$\bullet\   \mathbf{j} = \mathbf{5}$] $\;\Rightarrow \;\bS$ is term equivalent to a 2-element semilattice.
\end{itemize} %enumerate}
\end{proposition}
\begin{proof}
  This is a combination of~\cite[Proposition~3.1]{MR2504025} and~\cite[Theorem~6.1]{MR1191235}.
\end{proof}

We conclude this section with a result that will be useful in Section~\ref{sec:freese-valer-lemm}.
%% \rsf{did not see how FV 09 lemma 3.3 was used; eliminating it}
\begin{corollary}[{\protect \cite[Corollary~2.2]{Freese:2009}}]
  \label{cor:2.2}
  Let $\alg{A}$ be a finite idempotent algebra and $T$ an order ideal in the
  lattice of types. Then $\bbV(\alg{A})$ omits $T$ if and only if $\sansS(\alg{A})$ does.
  %% In particular, $\bbV(\alg{A})$ omits 1 and 2 if and only if $\sansS(\alg{A})$ omits 1 and 2.
\end{corollary}

%% \draftsecskip
%%%%%%%%%%%%%%%%%%%%%%%%%%%%%%%%%%%%%%%%%%%%%%%%%%%%%%%%%%%%%%%%%%%%%%%%%
%% \draftbreak

\section{The Characterization}
\label{sec:freese-valer-lemm}
In~\cite{Freese:2009}, Corollary~\ref{cor:2.2} is the starting point of the
development of a polynomial-time algorithm that determines if a given finite
idempotent algebra generates a congruence modular variety.

%% The following lemma ties in with the previous proposition and will be used
%% in Sec. 6.
%% \begin{lemma}[Lemma 2.3~\cite{Freese:2009}]
%%   Let $\alg{A}$ be a finite idempotent algebra and let $\bS \in \sansH \sansS(\alg{A})$
%%   be strictly simple. Then there are elements $a, b \in A$ such that, if
%%   $\alg{B} = \Sg^{\alg{A}} (a, b)$, then $1_B = \Cg^{\alg{B}} (a, b)$ and is join irreducible
%%   with unique lower cover $\rho$ such that $\bS = \alg{B}/\rho$.
%% \end{lemma}
%% \begin{proof}
%%   Choose $\alg{B} \in \sansS (\alg{A})$ as small as possible having $\bS$ as a homomorphic image,
%%   say $\bS = \alg{B}/\rho$. We claim that if $a, b \in B$ with
%%   $(a, b) \in \notin \rho$ then they generate $\alg{B}$. To
%%   see this, let $\alg{B}'= \Sg^{\alg{B}} (a, b)$ and let $h$ be the quotient map from B to S with kernel
%%   ρ. Then h(B  ) is a non-trivial subuniverse of S and so must equal S. Thus B  = B.
%%   Now let a, b ∈ B with (a, b) ∈
%%   / ρ. Since the block of Cg B (a, b) containing a
%%   and b is a subuniverse of B then from the previous paragraph, we conclude that
%%   Cg B (a, b) = 1 B and that ρ is its unique lower cover.
%% \end{proof}

According to the characterization
in~\cite[Chapter~8]{HM:1988} of locally finite congruence modular (resp.,
distributive) varieties, a finite algebra $\alg{A}$ generates a congruence modular
(resp., distributive) variety $\sV$ if and only if the typeset
of $\sV$ is
contained in $\{\atyp, \btyp, \ltyp\}$ (resp., $\{\btyp, \ltyp\}$)
and all minimal sets of prime
quotients of finite algebras in $\sV$ have empty
tails~\cite[Def.~2.15]{HM:1988}. (In the distributive
case the empty tails condition is equivalent to the minimal sets all having exactly
two elements.)

It follows from Corollary~\ref{cor:2.2} and Proposition~\ref{prop:2.1}
that if $\alg{A}$ is idempotent then one can
test the first condition---omitting
types \utyp, \styp (resp., \utyp, \atyp, \styp)---by searching
for a 2-generated subalgebra of $\alg{A}$ whose typeset is
not contained in
$\{\atyp, \btyp, \ltyp\}$ (resp., $\{\btyp, \ltyp\}$). It is proved
in~\cite[Section~6]{Freese:2009} that this
test can be performed in polynomial-time---that is, the running
time of the test is bounded by a polynomial function of the size of $\alg{A}$.
The main tools developed to this end are presented
in~\cite[Section~3]{Freese:2009} as a sequence of
lemmas that are used %enable the authors
to prove the following:
if $\alg{A}$ is finite and idempotent, and if
$\mathcal V = \bbV(\alg{A})$ omits types \utyp and \styp,
then to test for the existence of nonempty tails
in $\sV$ it suffices to look for them
in the 3-generated subalgebras of $\alg{A}^2$.
%% More specifically, the authors assume that the type set of $\bbV(\alg{A})$ contains no 1's
%% and no 5's, and under this
%% assumption they prove that nonempty tails either do not occur in $\bbV(\alg{A})$,
%% or they occur in 3-generated subalgebras of $\alg{A}^2$.
In other words, either there are no nonempty tails
or else there are nonempty tails that are easy to find
(since they occur in a 3-generated subalgebra of $\alg{A}^2$).
It follows that Problem~\ref{prob:2} has a positive answer:
deciding whether or not a finite idempotent algebra generates a congruence
modular variety is tractable.%
%\footnote{That is, there are positive integers
%$C, n$, and an algorithm that takes
%a finite idempotent algebra $\alg{A}$ as input and decides
%in at most $C|\alg{A}|^n$ steps whether $\bbV(\alg{A})$ is congruence modular.
%Here $|\alg{A}|$ denotes the number of bits required to encode
%the algebra $\alg{A}$.}
% polynomial-time algorithm to decide,
%for a finite idempotent algebra $\alg{A}$,
% whether $\bbV(\alg{A})$ is congruence modular.}

Our goal is to use the same strategy to solve Problem~\ref{prob:1}.
As such, we revisit each of the lemmas in Section 3 of \cite{Freese:2009},
and consider whether an analogous result can be proved under
modified hypotheses.
Specifically, we retain the assumption that the type set of $\bbV(\alg{A})$
omits \utyp, but we drop the assumption that it omits \styp.
We will  prove that, under these circumstances, % MAV: dropped  "we will attempt to prove"
either there are no type \atyp tails in $\bbV(\alg{A})$ (so the latter has a difference term), or else type \atyp tails can be found ``quickly,''
as they will occur in a 3-generated subalgebra of $\alg{A}^2$.
Where possible, we will relate our new results to
analogous results in~\cite{Freese:2009}.

\subsection{Notation}
Throughout we let $\nn$ denote the set $\{0,1,\dots, n-1\}$ and
we take $\sS$ to be a finite set of finite,
similar, idempotent algebras that is closed under the taking of
subalgebras, and we assume that the type set of
$\sV = \bbV(\sS)$ omits \utyp (but may include \styp).
If there exists a finite algebra in $\sV$ having a type \atyp\ minimal
set with a nonempty tail---in which case we say that
``$\sV$ has type \atyp tails''---then,
by standard results in \tct (see~\cite{HM:1988}),
at least one such algebra appears as a subalgebra of a finite product of
elements in $\sS$.
So we suppose that some finite algebra
$\alg{B}$ in $\sV$ has a prime quotient of type~\atyp with
minimal sets that have
nonempty tails and show that there is a 3-generated
subalgebra of the
product of two members of $\sS$ with this property (which can be found ``quickly'').

Since $\sS$ is closed under the taking of subalgebras,
we may assume that the algebra $\alg{B}$ from the previous paragraph is a subdirect
product of a finite number of members of $\sS$. Choose $n$ minimal such that for
some $\alg{A}_0$, $\alg{A}_1$, $\dots$, $\alg{A}_{n-1}$ in $\sS$, there is a subdirect
product $\alg{B} \sdp \prod_{\nn} \alg{A}_i$
that has a prime quotient of type~\atyp\ whose minimal sets have
nonempty tails.
Under the assumption that $n > 1$ we will prove that $n = 2$.

For this $n$, select the $\alg{A}_i$ and $\alg{B}$ so that $|B|$ is as small as possible.
Let $\alpha \prec \beta$ be a prime quotient of $\alg{B}$
of type~\atyp\ whose minimal sets have
nonempty tails, and choose $\beta$ minimal with respect to this property.
By~\cite[Lemma 6.2]{HM:1988}, this implies $\beta$ is join
irreducible and $\alpha$ is its unique subcover.
Let $U$ be an $\<\alpha, \beta\>$-minimal set.

For $i < n$, we let $\rho_i$ denote the kernel of the projection
of $\alg{B}$ onto $\alg{A}_i$, so $\alg{B}/\rho_i \cong \alg{A}_i$.
For a subset $\sigma \subseteq \nn$, define
\[
\rho_\sigma := \bigwedge_{j\in \sigma} \rho_j.
\]
Consequently,
$\rho_{\nn} = \bigwedge_{j\in \nn}\rho_j = 0_{B}$.
%% \marginnote{wjd: I don't see why join in (3.1) is $1_B$... it's probably wrong.}[3cm]
%% \begin{equation}
%%   \label{eq:2}
%%   \rho_{\nn} = \bigwedge_{j\in \nn}\rho_j = 0_{B} \quad \text{ and } \quad
%%   \bigvee_{j\in \nn}\rho_j =1_B. %% \qquad
%% \end{equation}
By minimality of $n$ we know that the intersection of any  proper subset of the
$\rho_i$, $0\leq i < n$ is strictly above $0_B$.  Thus,
$0_B < \rho_\sigma < 1_B$ for all
$\emptyset \subset \sigma\subset \nn$.
(By $\subset$ we mean \emph{proper} subset.)

The next four lemmas assume the context above, which for convenience
we will denote by $\Gamma$; that is,
``\emph{Assume $\Gamma$}'' will mean ``Assume $\sS$, $\alg{B}$, $n$, $\{\alg{A}_i \colon i < n\}$,
$\alpha$, $\beta$, $U$, and $\rho_\sigma$ are as described above.''

\begin{lemma}[{\protect cf.~\cite[Lemma~3.1]{Freese:2009}}]
\label{lem:fv_3-1}
Assume $\Gamma$. If\/ $0$, $1 \in U$, if $(0,1) \in \beta \mysetminus \alpha$, and if
$t$ belongs to the tail of $U$, then $\beta$ is the congruence of $\alg{B}$
generated by the pair $(0,1)$, and $\alg{B}$ is generated by $\{0, 1, t\}$.
\end{lemma}
%and let $N$ be an
% $(\alpha, \beta)$-trace of $U$. Let 0 and 1 be
%two distinct members of $N$ with $(0, 1) \notin \alpha$.
This follows from the same argument used to prove~\cite[Lemma~3.1]{Freese:2009}, using the fact that since $\beta$ is abelian over $\alpha$ it follows that $\gamma$ over $\delta$ that appears in the proof will also have type 2.

\begin{lemma}[{\protect cf.~\cite[Lemma~3.2]{Freese:2009}}]
  \label{lem:fv_3-2}
Assume $\Gamma$. For every proper nonempty subset $\sigma$ of $\nn$,
  either $\beta \leq \rho_\sigma$ or $\alpha \join \rho_\sigma = 1_B$.
\end{lemma}
  This follows from the proof ~\cite[Lemma~3.2]{Freese:2009}.

\begin{lemma}\label{lem:nearperm}
Assume $\Gamma$.  If $\rho \in \Con(\alg{B})$ with $\beta \join \rho = 1_B$ then
  for all $v\in B$, and for all $b\in \body(U)$, we have
  $(v,b) \in (\beta \circ \rho) \cap (\rho \circ \beta)$.
\end{lemma}

%\smallskip
%% \newcommand\rhosig{\ensuremath{\rho_\sigma}}
\newcommand\rhosig{\ensuremath{\rho}}
\begin{proof}
\noindent Note that
$\beta \join \rho = 1_B$ implies
$\restr{\beta}{U} \join \restr{\rho}{U} = \restr{1_B}{U} = 1_U$,
since $U = e(B)$, for some idempotent unary polynomial~$e$.
Now, for all $x$, $y \in U$, if $x\in \body(U)$ and $y\in \tail(U)$, then
$(x,y) \notin \beta$.  Therefore,
$(x, y) \in  1_U = \restr{\beta}{U} \join \restr{\rho}{U}$ implies
there must be some $b' \in \body(U)$ and $t'\in \tail(U)$ such that
$b' \mathrel{\rho} t'$.

Now, let $d(x,y,z)$ be a pseudo-\malcev polynomial for $U$,
which exists by~\cite[Lemma~4.20]{HM:1988}.
Thus,
\begin{itemize}
\item $d(B,B,B) = U$
\item $d(x,x,x) = x$ for all $x\in U$
\item $d(x,x,y) = y = d(y,x,x)$ for all $x\in \body(U)$, $y \in U$.
\end{itemize}
Moreover, for all $c$, $d \in \body(U)$, the unary polynomials
$d(x,c,d)$, $d(c,x,d)$, and $d(c,d,x)$ are permutations on $U$.
If we now fix an arbitrary element $b\in \body(U)$ and
let $p(x) = d(x,b',b)$, then (see~\cite[Lemma~4.20]{HM:1988})
\begin{itemize}
\item  $p(U) = U$, since $d(x, b', b)$ is a permutation of $U$, %since $U$ is minimal,
\item $p(b') = d(b',b',b) = b \in \body(U)$, and
\item  $t:=p(t')\in \tail(U)$, since $t'\in \tail(U)$.
\end{itemize}
Since $(b',t') \in \rho$, we have $(b, t) = (p(b'), p(t')) \in \rho$.
Since $b$ is in the body, there is an element $b''$ in the body with
$(b,b'') \in \beta - \alpha$. By Lemma~\ref{lem:fv_3-1}, this implies
$\alg{B} = \Sg^{\alg{B}}(\{b, b'', t\})$.

Finally, if $v \in B$, then $v = s^{\alg{B}}(b,b'',t)$ for
some (idempotent) term $s$, so
\[
v = s^{\alg{B}}(b,b'',t)
\mathrel{\rho} s^{\alg{B}}(b,b'',b)
\mathrel{\beta} s^{\alg{B}}(b,b,b) = b,
\]
and
\[
v = s^{\alg{B}}(b,b'',t)
\mathrel{\beta}  s^{\alg{B}}(b,b,t)
\mathrel{\rho}s^{\alg{B}}(b,b, b)  = b.
\]
Therefore,
$(v,b) \in  (\beta \circ \rho) \cap (\rho \circ \beta)$.
Since $v \in B$ and $b\in \body(U)$ were aribitrary,
this completes the proof.
\end{proof}

\begin{lemma}[{\protect cf.~\cite[Lemma~3.3]{Freese:2009}}]\
  \label{lem:fv_3-3}
  Assume $\Gamma$.
  \begin{enumerate}[(i)]
    \item \label{item:6} There exists $i$
      such that $\alpha \join \rho_i = 1_B$
    \item \label{item:7} There exists $i$ such that
      $\alpha \join \rho_i < 1_B$.
  \end{enumerate}
\end{lemma}
\begin{proof}
  %% {\bf TODO:} fill in proof of  Lemma~\ref{lem:fv_3-3}.
  %%\begin{enumerate}[(i)]
  %%\item %\label{item:6} There exists $0\leq i< n $  such that $\alpha \join \rho_i = 1_B$
If item \eqref{item:6} failed, then by Lemma~\ref{lem:fv_3-2} we would
have $\beta \leq \rho_i$ for all $i$, and that
would imply $\beta = 0_B$.
  %%\item %\label{item:7} There exists $i$ such that $\alpha \join \rho_i < 1_B$.

To see \eqref{item:7}, assume that
$\alpha \join \rho_i = 1_B$ for all $i$.  Then $\beta \join \rho_i = 1_B$ for all $i$ as well.
Let $b \in \body(U)$ and $t \in \tail(U)$, and let $d(x,y,z)$
denote the pseudo-\malcev operation introduced in the proof of Lemma~\ref{lem:nearperm}. By
\cite[Lemma~4.25]{HM:1988},  $(b, d(b,t,t)) \notin \beta$.
We will arrive at a contradiction by showing that
$b = d(b,t,t)$. To accomplish this we will show that $d(b,t,t) \rho_i b$ for all $i$ and use that $\bigwedge \rho_i = 0_B$.

If $i \in \nn$ then by Lemma~\ref{lem:nearperm},
$(b,t) \in \beta \circ \rho_i$ so
there is an element $a \in B$ satisfying
$b\mathrel\beta a \mathrel\rho_i t$. By applying the idempotent
polynomial $e$ with $e(U) = U$, we have
$b\mathrel\beta e(a) \mathrel\rho_i t$, so we may
assume $a \in U$. But this puts $a \in \body(U)$,
since $a\mathrel \beta b$. Therefore,
\[
d(b,t,t) \mathrel\rho_i d(b,a,a) = b.
\]
Since this hold for every~$i$, $d(b,t,t) = b$.
\end{proof}

%\begin{lemma}[{\protect cf.~\cite[Lem.~3.3]{Freese:2009}}]
%\label{lem:fv_3-3x}
%There is exactly one~$i$ such that $\beta \not\le \rho_i$.
%\end{lemma}
%\begin{proof}
%If $n = 1$ then $\rho_1 = 0$ so the results holds. The lemma
%above shows there must there must be at least one $\rho_i \ge \beta$
%so let $\rho_1 \not\ge \beta$ and $\rho_2 \ge \beta$. If $n = 2$,
%the result holds. Otherwise let $\rho = \rho_1 \meet \rho_2$ and
%note it is clearly not above $\beta$. Since $\beta$ is
%join irreducible and $\alpha$ is its unique lower cover,
%$\rho \meet \beta \le \alpha$. But then $\alg B/\rho$ has a type~2
%tail, contradicting the minimality of~$n$.
%\end{proof}

\begin{theorem}[{\protect cf.~\cite[Theorem 3.4]{Freese:2009}}]\label{thm:fv_3-4}
Let $\sV$ be the variety generated by some finite set $\sS$ of finite,
idempotent algebras that is closed under taking subalgebras. If\/ $\sV$
omits type~\utyp\ and some finite member of $\sV$ has a prime quotient
of type~\atyp\
whose minimal sets have nonempty tails, then there is some
3-generated algebra $\alg B$ with this property that belongs to $\sS$ or is a subdirect
product of two algebras from $\sS$.
\end{theorem}
\begin{proof}
Choose $n > 0$, $\alg A_i \in \sS$, for $0 \le i \le n-1$ and $\alg B$
as described just before Lemma~\ref{lem:fv_3-1}. From Lemma \ref{lem:fv_3-1} we know that $\alg B$ is
3-generated. If $n > 1$ then by the previous lemma we can choose $i$
and $j < n$ with $\alpha \join \rho_i < 1_B$ and $\alpha \join \rho_j =
1_B$. If $n > 2$ then Lemma~\ref{lem:fv_3-2} applies to $\rho = \rho_i
\meet \rho_j$ and so we know that either $\beta \le \rho$ or $\alpha
\join \rho = 1_B$. This yields a contradiction as the former is not
possible, since then
\[
\alpha \join \rho_j \le \beta \join \rho_j \le \rho \join \rho_j = \rho_j < 1_B.
\]%$\beta \not\le \rho_j$ and
The latter can't hold
since  $\alpha \join \rho  = 1_B$ implies that $\alpha \join \rho_i = 1_B$.
So, the minimality of $n$ forces $n\le 2$ and the result follows.
\end{proof}

\begin{example}\label{alg:nodiff}
Let $\alg A$ be the (idempotent) algebra with universe $\{0,1,2, 3\}$ with basic operation $x\cdot y$ defined by:
  \[
  \begin{array}{c|cccc}
  \cdot&0&1&2&3\\\hline
  0&0&2&1&3\\
  1&2&1&0&3\\
  2&1&0&2&3\\
  3&3&0&0&3
  \end{array}
  \]
It can be checked that $\alg A$ is a simple algebra of type~\btyp and that no subalgebra of $\alg A$ has a prime quotient of type~\atyp whose minimal sets have nonempty tails.  It can also be checked that the subalgebra of $\alg A^2$ generated by $\{(0,0), (1,0), (0,3)\}$ does have such a prime quotient.
Note that these calculations can be carried out with the aid of the Universal Algebra Calculator~\cite{UACalc}.
This demonstrates that in general one must look for nonempty tails of minimal sets of type~\atyp in the square of a finite idempotent algebra and not just in the subalgebras of the algebra itself.

We note that, since $\alg A$ is simple of type~\btyp, the ternary projection operation $p(x,y,z ) = z$ is a difference term operation for $\alg A$.  We also note that the term operation $(x\cdot y) \cdot(y \cdot x)$ of $\alg A$ is idempotent and commutative, so the variety generated by $\alg A$ omits type~\utyp (see Lemma 9.4 and Theorem 9.6 of~\cite{HM:1988}).  So, this example also demonstrates that for finite idempotent algebras, having a difference term operation and generating a variety that omits type~\utyp does not guarantee the existence of a difference term for the variety.
\end{example}

\begin{corollary}\label{cor:diffterm}
  Let $\alg A$ be a finite idempotent algebra.  Then $\bbV(\alg A)$ has a difference term if and only if
  \begin{itemize}
  \item $\sansH \sansS (\alg{A})$ does not contain an algebra that is term equivalent to the 2-element set and
      \item no 3-generated subalgebra of $\alg A^2$ has a prime quotient of type~\atyp whose minimal sets have nonempty tails.
  \end{itemize}
\end{corollary}

\begin{proof}
  This is just a combination of Proposition~\ref{prop:2.1} and Theorems~\ref{thm:KearnesThm} and~\ref{thm:fv_3-4}.
\end{proof}

The next theorem essentially gives an algorithm to decide if a finitely
generated, idempotent variety has a difference term. In the next
section, we will show that the algorithm runs in polynomial-time.

In \cite{KearnesKiss1999}, Kearnes and Kiss show there is a close connection
between $\alpha \cov \beta$ being the critical interval of
a pentagon and $\la \alpha,\beta\ra$-minimal sets
having nonempty tails.
By \cite[Theorem 2.1]{KearnesKiss1999}, the minimal sets
of a prime critical interval of a pentagon have nonempty tails, provided
the type is not~\utyp. In the other direction, if the
$\la \alpha,\beta\ra$-minimal sets have nonempty tails,
then there is a pentagon in the congruence lattice of a subalgebra of
$\alg A^2$ with a prime critical interval of the same type.
This connection between minimal sets with tails and
pentagons is important for us since we do not have a polynomial time algorithm
for finding an $\la \alpha,\beta\ra$-minimal set.

If $\alg B$ is a subalgebra of $\alg A^2$ and $\theta$ is a congruence
of $\alg A$, then we define $\thetao \in \con (\alg B)$ by
$(x_0,x_1) \mathrel{\thetao} (y_0,y_1)$ iff
$x_0 \mathrel{\theta} y_0$. We define $\theta_1$ similarly.
In case $\theta = 0_{\alg A}$, the least congruence,
we use the notation $\rho_0$ and $\rho_1$ instead of
$0_0$ and $0_1$. Of course $\rho_0$ and $\rho_1$ are the kernels
of the first and second projections of $\alg B$ into~$\alg A$.

\begin{theorem}\label{thm:algorithm}
Let $\alg A$ be a finite idempotent algebra and let $\sV$ be the variety
it generates. Then $\sV$ has a difference term if and only if the
following conditions hold:
\begin{enumerate}

   % $\alg A$ has a Taylor term.
   \item \label{it:1}$\sV$ omits TCT-type~\utyp.
  \item \label{it:2}
    There do not exist $a$, $b$, $c\in A$
    satisfying the following, where
    % $\alg B$ is the subalgebra of $\alg A$ generated
    % by $a$, $b$ and $c$,  $\beta = \Cga a b B$, and
    $\alg B := \Sg^{\alg A}(a, b, c)$ and
    % $\alg C$ is the subalgebra of $\alg B^2$ generated by
    % $(a,b)$, $(a,c)$, $(b,c)$ and the diagonal of $B$,
    $\alg C:=\Sg^{\alg{B}^2}\bigl(\{(a,b), (a,c), (b, c)\}
                        \cup 0_{\alg B}\bigr)$:
    \begin{enumerate}
      \item \label{it:2a}
        $\beta := \Cga a b B$ is join irreducible with lower cover $\alpha$,
      \item \label{it:2b}
        $((a,b),(b,b)) \notin (\alpha_0 \meet \alpha_1) \join \Cga {(a,c)} {(b,c)} C$,  and
      \item \label{it:2c}
        $[\beta,\beta] \le \alpha$.
    \end{enumerate}

  \item \label{it:3}
    There do not exist $x_0$, $x_1$, $y_0$, $y_1\in A$ satisfying
    the following, where $\alg B$ is the subalgebra of
    $\alg A \times \alg A$ generated by $0 := (x_0, x_1)$, $1 := (y_0,x_1)$,
    and $t := (x_0,y_1)$:
    % $\rho_0$ is the kernel of the first projection:
    \begin{enumerate}
      \item \label{it:3a}
        $\beta := \Cga01{B}$ is join irreducible with lower cover $\alpha$,
      \item \label{it:3b}
        $\rho_0 \join \alpha = 1_{\alg B}$, and
      \item \label{it:3c}
        the type of $\beta$ over  $\alpha$ is~\atyp.
    \end{enumerate}
  \end{enumerate}
\end{theorem}

\begin{proof}
First assume $\sV$ has a difference term.
Then~(\ref{it:1}) holds by Theorem~\ref{thm:KearnesThm}.
If~(\ref{it:2}) fails then there are $a$, $b$ and $c\in A$ such
that the conditions specified in~(\ref{it:2}) hold.
In particular,~(\ref{it:2c}) holds and implies that
$\typ\<\alpha, \beta\> \subseteq \{\utyp, \atyp\}$,
so by~(\ref{it:1}) the type of $\<\alpha, \beta\>$ is~\atyp.
Let
\begin{align*}
\delta &:= (\alpha_0 \meet \alpha_1) \join \Cga {(a,c)} {(b,c)} C,
            \text{ and }\\
\theta &:= \delta \join \Cga {(a,b)}{(b,b)} C.
\end{align*}
By its definition, $\delta \nleq \alpha_0$,
so by~(\ref{it:2b}), $\alpha_0\meet\alpha_1 < \delta < \theta \le \beta_0$.
Since $C$ contains $0_B$, the diagonal of $B$, the coordinate
projections are onto, so $\alpha_0 \cov \beta_0$ and this interval
has type~\atyp.
From this it follows that $\alpha_0 \join \delta = \beta_0$.
Since $\theta \le \alpha_1$, we have $\alpha_0 \meet \theta = \alpha_0\meet\alpha_1$.
Hence
\[
\{\alpha_0\meet\alpha_1, \delta, \theta, \alpha_0, \beta_0 \}
\]
forms a pentagon.
Since $[\beta_0,\beta_0] \leq \alpha_0$, we have
$[\theta,\theta] \leq \alpha_0\meet\alpha_1 < \delta$,
so there is a congruence $\delta'$ such that
$\delta\leq \delta' \cov \theta$ and $\<\delta', \theta\>$
has type~\atyp.
As mentioned in the discussion above, this implies the
$\la \delta',\theta\ra$-minimal sets have nonempty tails, contradicting
Theorem~\ref{thm:KearnesThm}.

Now suppose that~(\ref{it:3}) fails. Then the conditions imply
\[
\{0_{\alg B}, \alpha, \beta, \rho_0, 1_{\alg B} \}
\]
is a pentagon whose critical prime interval has type~\atyp. This
leads to a contradiction in the same manner as above.

For the converse assume that $\sV$ does not have a difference
term.
We want to show that at least one of~(\ref{it:1}), (\ref{it:2}) or~(\ref{it:3}) fails.
%Assume all three hold.
By Theorem~\ref{thm:KearnesThm} either $\sV$ admits \tct-type~\utyp, in which case~(\ref{it:1}) fails, or there is a finite
algebra $\alg B \in \sV$ and a join irreducible
$\beta \in \op{Con}(\alg B)$ with lower cover
$\alpha$ such that the type of $\<\alpha, \beta\>$ is~\atyp\
and the $\la\alpha,\beta\ra$-minimal sets have
nonempty tails. Let $U$ be one of these minimal sets.

We may assume $\alg B$ is minimal in the same manner as with the
above lemmas (with $\sS$ being the set of subalgebras of $\alg A$).
By Lemma~\ref{lem:fv_3-1}
we have that $\alg B$ is generated by any $0$, $1$, and $t$ in
$U$ such
that $\beta = \Cga 01{B}$ and $t$ is in the tail. By
Theorem~\ref{thm:fv_3-4}, $\alg B$ is either
in $\sS$ or is a subdirect product of two members of $\sS$.

Assume $\alg B$ is a subalgebra of $\alg A$. Taking
$a=0$, $b=1$ and $c=t$, we claim the conditions specified
in~(\ref{it:2}) hold for these three elements.
Since the type of $\beta$ over $\alpha$ is \atyp,~(\ref{it:2c})
holds and we already have that~(\ref{it:2a}) holds.
(\ref{it:2b}) holds by~\cite[Theorem~2.4]{KearnesKiss1999}.
So this choice of $a$, $b$, and $c$ witness that~(\ref{it:2}) fails.

Now assume $\alg B$ is not in $\sS$ but is a subdirect
product of two members of $\sS$.
Then
by Lemma~\ref{lem:fv_3-3} we may
assume $\rho_0 \join \alpha = 1_{\alg B}$ and
$\rho_1 \join \alpha < 1_{\alg B}$. By Lemma~\ref{lem:fv_3-2}
we have $\rho_1 \ge \beta$.

This implies that $0$ and $1$ have the same second coordinate; that is,
$0 = (x_0,x_1)$ and $1 = (y_0,x_1)$ for some $x_0$, $y_0$ and $x_1\in A$.
By Lemma~\ref{lem:nearperm}, $(0,t) \in \rho_0 \circ \beta$
so $0 \mathrel {\rho_0} t' \mathrel{\beta} t$ for some $t' \in B$. Let $U = e(B)$
where $e$ is an idempotent polynomial. Then
$0 \mathrel {\rho_0} e(t') \mathrel{\beta} t$. This gives
that $e(t')$ is in the tail of $U$ and
$0 \mathrel{\rho_0} e(t')$. We can
replace $t$ by $e(t')$, and so assume that
$0 \mathrel{\rho_0} t$.
Since $0 = (x_0,x_1)$ then $t = (x_0,y_1)$ for some $y_1\in A$.
Now
$x_0$, $y_0$, $x_1$ and $y_1$ witness that~(\ref{it:3}) fails.
\end{proof}

\section{The Algorithm and its Time Complexity}
\label{sec:algorithm-its-time}
If $\alg A$ is an algebra with underlying set (or universe) $A$,
we let $|\alg A| = |A|$ be the cardinality of
$A$ and $||\alg A||$ be the \emph{input size}; that is,
\[
||\alg A|| = \sum_{i=0}^r k_i n^i
\]
where, $k_i$ is the number of basic operations of arity~$i$ and $r$
is the largest arity of the basic operations of $\alg A$. We set
$n = |\alg A|$ and $m = ||\alg A||$.
%\begin{align*}
%n &= |\alg A|  \qquad m = ||\alg A|| \\
%r &= \text{the largest arity of the operations of $\alg A$}
%\end{align*}

%Throughout this section we let $c$ denote a constant independent of
%these parameters.

\begin{proposition}\label{speedprop}
Let $\alg A$ be a finite algebra with the parameters above. Then
there is a constant $c$ independent of these parameters
such that:
\begin{enumerate}
\item \label{speed1} If $S$ is a subset of $A$,
then $\Sg^{\alg A}(S)$ can be computed
in time
\[
c\, r\,||\Sg^{\alg A}(S)|| \le c\, r\,||\alg A|| = crm
\]
\item \label{speed2} If $a$, $b \in A$, then $\Cga a b A$ can be
computed in
$c\, r\, ||\alg A|| = crm$ time.
\item \label{speed3}
If $\alpha$ and $\beta$ are congruences of $\alg A$,
then $[\alpha,\beta]$ can be computed in time $crm^4$.
If $\alg A$ has a so-called \emph{Taylor term}
(see e.g.~\cite{MR3076179} or~\cite{HM:1988} for the definition),
and if $[\alpha,\beta] = [\beta,\alpha]$,
%%\mav{Should we define what a Taylor term is or replace this the omitting type 1 equivalence?}
then $[\alpha,\beta]$ can be computed in time $c(rm^2 + n^5)$.
In particular, $[\beta,\beta]$ can be computed in this time.
\end{enumerate}
\end{proposition}

\begin{proof}
For the first two parts see~\cite[Proposition~6.1]{Freese:2009}.
To see the third part we first describe a method to compute
$[\alpha,\beta]$.

Following the notation of~\cite{FreeseMcKenzie1987}, we write
elements of $\alg A^4$ as $2 \times 2$ matrices, and
let $M(\alpha,\beta)$ be the subalgebra of $\alg A^4$ generated by
the elements of the form
\[
\begin{bmatrix}
a & a\\
a' & a'\\
\end{bmatrix}
\quad \text{and} \quad
\begin{bmatrix}
b & b'\\
b & b'\\
\end{bmatrix}
\]
where $a \mathrel\alpha a'$ and $b \mathrel\beta b'$. Then
by definition
$[\alpha,\beta]$ is the least congruence $\gamma$ such that
\begin{equation}\label{commprop}\text{if }
\begin{bmatrix}
x & y\\
u & v\\
\end{bmatrix}
\text{ is in $M(\alpha,\beta)$ and $x \mathrel\gamma y$,
then $u \mathrel\gamma v$.}
\end{equation}

Let $\delta = [\alpha,\beta]$. Clearly, if
$\begin{bmatrix}
x & x\\
u & v\\
\end{bmatrix}$ is in  $M(\alpha,\beta)$, then $u \mathrel\delta v$.
Let $\delta_1$ be the congruence generated by the $(u,v)$'s so
obtained. Then $\delta_1 \le \delta$.

We can now define $\delta_2$ as the congruence generated by the pairs
$(u,v)$ arising as the second row of members of $M(\alpha,\beta)$,
where the elements of the first row are $\delta_1$ related.
More precisely, we inductively define $\delta_0 = 0_{\alg A}$ and
\begin{equation}\label{delta}
\delta_{i+1} = \operatorname{Cg}^{\alg A}\bigg(\bigg\{(u,v) :
\begin{bmatrix}
x & y \\
u & v\\
\end{bmatrix}
\in M(\alpha,\beta) \text{ and } (x,y) \in \delta_i\bigg\}\bigg)
\end{equation}
Clearly, $\delta_1 \le \delta_2 \le \cdots \le
\delta$ and so $\Join_i \delta_i \le \delta$.

Now $\Join_i \delta_i$ has the property~\eqref{commprop} of the
definition of $[\alpha,\beta]$,
and hence $\delta \le \Join_i \delta_i$, and thus they are equal.

So to find $[\alpha,\beta]$ when $\alg A$ is finite, we find
$M(\alpha,\beta)$ and then compute the $\delta_i$'s, stopping when
$\delta_i = \delta_{i+1}$, which will be $[\alpha,\beta]$ of course.
The time to compute $M(\alpha,\beta)$ is bounded by $crm^4$ by
part~\eqref{speed1}. Suppose we have computed $\delta_i$. To
compute $\delta_{i+1}$ we run through the at most $n^4$ matrices
in
$\begin{bmatrix}
x & y \\
u & v\\
\end{bmatrix} \in M(\alpha,\beta)$. If $x$ and $y$ are in the same
block, we join the block containing $u$ and the one containing~$v$
into a single block. By the techniques of \cite{Freese2008},
this can be done in constant time. Now we take the congruence generated
by this partition. So, by \eqref{speed2} the time to compute
$\delta_{i+1}$ from $\delta_i$ is $c(n^4 + rm)$. Since the congruence
lattice of $\alg A$ has length at most $n-1$, there are at most~$n$
passes. So the total time is at most a constant times
$rm^4 + n(n^4 + rm) = rm^4 + n^5 + nrm$.
Since the commutator is trivial in unary algebras, we may
assume $m \ge n^2$, and thus the time is bounded by a
constant times $rm^4$, proving the first part of~\eqref{speed3}.
This procedure for calculating the commutator is part
of Ross Willard's thesis~\cite{MR2637477}.

To see the other part,
let $\alg A(\alpha)$ be the subalgebra
of $\alg A \times \alg A$ whose components are $\alpha$ related. If we
view the elements of $\alg A(\alpha)$ as column vectors, then the
elements of $M(\alpha,\beta)$ can be thought of as pairs of elements
of $\alg A(\alpha)$, that is, as a binary relation on $\alg A(\alpha)$.
Define $\Delta_{\alpha,\beta}$ to be the congruence on $\alg A(\alpha)$
generated by this relation, that is,
the transitive closure of this relation.
Clearly $M(\alpha,\beta) \subseteq
\Delta_{\alpha,\beta}$. So, if in the algorithm above we used
$\Delta_{\alpha,\beta}$ in place of $M(\alpha,\beta)$, the result
would be at least $\delta = [\alpha,\beta]$. So, if we knew that
\begin{equation}\label{deltaStatement}
\text{whenever
$\begin{bmatrix}
x & y \\
u & v\\
\end{bmatrix} \in \Delta_{\alpha,\beta}$ and $(x,y) \in \delta$,
then $(u,v) \in \delta$},
\end{equation}
then this modified procedure would also
compute~$\delta$.

While \eqref{deltaStatement} is not true in general even if
$\alg A$ has a Taylor term, it is true when
$\alg A$ has a Taylor term and $[\alpha,\beta] = [\beta,\alpha]$.
So assume $\alg A$ has a Taylor term and $[\alpha,\beta] = [\beta,\alpha]$.
Under these conditions the commutator agrees with the
linear commutator, that is, $[\alpha,\beta] = [\alpha,\beta]_\ell$ by
Corollary~4.5 of~\cite{MR1663558}.
Suppose
$
\begin{bmatrix}
a&b\\c&d
\end{bmatrix}\in \Delta_{\alpha,\beta}$ and $(a,b)\in \delta$.
Then, since
$\Delta_{\alpha,\beta}$ is the transitive closure of
$M(\alpha,\beta)$, there are
elements $a_i$ and $c_i$ in $A$,
$i = 0, \ldots, k$, with $a_0 =a$,  $c_0 = c$, $a_k = b$ and
$c_k = d$, such that
$
\begin{bmatrix}
a_i&a_{i+1}\\c_i&c_{i+1}
\end{bmatrix}\in M(\alpha,\beta)$.

Now the linear commutator is $[\alpha^*,\beta^*]\big\vert_{A}$,
where $\alpha^*$ and $\beta^*$ are congruences on an expansion
$\alg A^*$ of $\alg A$ such that $\alpha \subseteq \alpha^*$
and $\beta \subseteq \beta^*$; see Lemma~2.4 of~\cite{MR1663558}
and the surrounding discussion.

Moreover $M(\alpha,\beta) \subseteq M(\alpha^*,\beta^*)$, the latter
calculated in $\alg A^*$, because the generating matrices of
$M(\alpha^*,\beta^*)$ contain those of $M(\alpha,\beta)$, and
the operations of $\alg A$ are contained in the operations
of $\alg A^*$.  So
$
\begin{bmatrix}
a_i&a_{i+1}\\c_i&c_{i+1}
\end{bmatrix}\in M(\alpha^*,\beta^*)$.
By its definition $\alg A^*$ has a \malcev\ term, and
consequently $M(\alpha^*,\beta^*)$ is transitive as
a relation on $\alg A(\alpha^*)$. Thus
$
\begin{bmatrix}
a&b\\c&d
\end{bmatrix}\in M(\alpha^*,\beta^*)$, and hence,
$(c,d) \in [\alpha^*,\beta^*]\big\vert_A = [\alpha,\beta]$.

Now, since $\Delta_{\alpha,\beta}$ is a congruence on
$\alg A(\alpha)$, it can be computed in time
$crm^2$ by part~\eqref{speed2}. The result follows.
\end{proof}

%The below needs fixing: and def of Delta is needs to be changed to the
%Comm book and the property here is from (3.3)(1).

%We claim that if in \eqref{delta} we
%we replace $M(\alpha,\beta)$ by
%$\Delta_{\alpha,\beta}$, the resulting $\delta_i$'s are unchanged.
%To see this let the $\delta_i$'s be defined by~\eqref{delta} and
%let the $\delta'_i$'s be defined similarly but using
%$\Delta_{\alpha,\beta}$ in place of $M(\alpha,\beta)$.
%Since $M(\alpha,\beta) \subseteq \Delta_{\alpha,\beta}$, $\delta_i \le %\delta'_i$.

% define \Delta, ref that it is the transitive closure of the matrices
% and that if we replace the matrices with Delta in the above process
% we get the same result.

\begin{theorem}\label{thm:time}
Let $\alg A$ be a finite idempotent algebra with parameters as
above.
Then one can determine if $\bbV(\alg A)$ has a difference
term in time $c(rn^4m^4 + n^{14})$.
\end{theorem}
%% \rsf{See if we can omit $n^{14}$.}

\begin{proof}
Theorem~\ref{thm:algorithm} gives a three-step
algorithm to test
if $\bbV (\alg A)$ has a difference term.
The first step is to test if $\bbV (\alg A)$ omits type~\utyp. This
can be done in time $crn^3m$
by~\cite[Theorem~6.3]{Freese:2009}.

Looking now at part (3) of Theorem~\ref{thm:algorithm},
there are several things that have to be constructed.
By Proposition~\ref{speedprop}, all
of these things can be constructed in time $crm^2$ and
parts~(a) and~(b) can be executed in this time or less.
For part~(c) we need to test if the type of $\beta$
over $\alpha$ is \atyp. Since at this point in the
algorithm we know that $\alg A$ omits type ~\utyp,
we can test if the type is \atyp\ by testing if
$[\beta,\beta] \le \alpha$. By Proposition~\ref{speedprop}
%% \mav{could we test for 2-snags in $\beta\setminus\alpha$ instead?  Is this any faster? (I'm pretty sure that we've already considered this, so probably not.)}
%% Ralph's response as to whether 2-snags might be faster:
%%   The best I know for an algebra A is m^3, while using the commutator is m^2 + n^5.
%%   Since the questions we are concerned with are trivial for unary algebras,
%%   m \ge n^2. So 2-snags aren’t as good.
this can be done in time $c(rm^4 + n^{10})$.
Since we need to do
this for all $x_0$, $x_1$, $y_0$ and $y_1$, the total
time for this step is at most $crn^4m^4 + n^{14}$.
A similar analysis applies to part~(2) and shows that it
can be done in time $crn^3m^2$. Since $crn^4m^4$ dominates
the other terms, the bound of the theorem holds.
\end{proof}

%%%%%%%%%%%%%%%%%%%%%%%%%%%% wjd: NEW SECTION  %%%%%%%%%%%%%%%%%%%%%%%%%%%%%%
\section{Difference Term Operations}
Above we addressed the problem of deciding the existence of a difference term
for a given (idempotent, finitely generated) variety.  In this section we are
concerned with the practical problem of finding a difference term
\emph{operation} for a given (finite, idempotent) algebra.
We describe algorithms for
\begin{enumerate}
\item \label{item:a} deciding whether a given finite idempotent algebra
has a difference term operation, and
\item \label{item:b} finding a difference term operation
for a given finite idempotent algebra.
\end{enumerate}
Note that Theorem~\ref{thm:time} gives a polynomial-time algorithm
for deciding whether or not the variety $\bbV(\alg A)$ generated by a
finite idempotent algebra $\alg A$ has a difference term.
If we run that algorithm on input $\alg A$, and if the observed
output is ``Yes'', then of course we have a positive answer to decision
problem~(\ref{item:a}).  However, a negative answer returned by the
algorithm only tells us that $\bbV(\alg A)$ has no difference term.
It does not tell us whether or not $\alg A$ has a difference term operation.  Example~\ref{alg:nodiff} provides a finite idempotent algebra that has a difference term operation such that the variety that it generates does not have a difference term.

%\mav{Again, improve example (see Ex. 9).}
%\begin{example}
%  Let $\alg A$ be the idempotent algebra with universe $\{0,1,2\}$ with basic operation $x\cdot y$ defined by:
%  \[
%  \begin{array}{c|ccc}
%  \cdot&0&1&2\\\hline
%  0&0&0&1\\
%  1&0&1&1\\
%  2&0&2&2
%  \end{array}
%  \]
%  Then $\alg A$ is a simple algebra of type~\ltyp and so the ternary projection operation $p(x,y,z ) = z$ is a difference term operation for $\alg A$. On the other hand, since $\{1,2\}$ is a subuniverse of $\alg A$ and $x\cdot y = x$ on this subset then the variety generated by $\alg A$ cannot have a difference term.
%\end{example}

%\wjd{insert example: $\alg A$ has a diff term op, but
 % $\bbV(\alg A)$ has no diff term.}

\medskip

In this section we present solutions to problems~(\ref{item:a}) and~(\ref{item:b})
using different methods than those of the previous sections.
In Subsection~\ref{sec:algor-1} we give a polynomial-time algorithm
for deciding whether a given idempotent algebra $\alg A$ has a difference term operation.
In Subsection~\ref{sec:comp-diff-term} we address problem~(\ref{item:b})
by presenting an algorithm for constructing  the operation table of a difference term operation.
%% However, this algorithm does not run in polynomial-time and, at the time
%% of this writing, we do not know of a more efficient algorithm for constructing a
%% difference term operation, even when such an operation is known to exist.
%Also, we have not yet proved the existence of a tractible algorithm for constructing %difference term operations.

\subsection{Local Difference Terms}
\label{sec:local-diff-terms}
In~\cite{MR3239624},
Ross Willard and the third author define %% an \defin{$\bA$-triple for $\bp$}
%% to be a triple $(a,b,i)$ such that $a, b \in A$ and
%% $p_i(a,b,b) = p_{i+1}(a,a,b)$. They use this to define
a ``local Hagemann-Mitschke sequence'' which they use as the basis of
an efficient algorithm for deciding for a given $n$ whether an idempotent
variety is $n$-permutable.
In~\cite{MR3109457}, Jonah Horowitz introduced similar
local methods for deciding when a given variety satisfies
certain \malcev conditions.
Inspired by these works, we now define a ``local difference term
operation'' and use it to develop a polynomial-time algorithm for deciding
the existence of a difference term operation.

Start with a finite idempotent algebra, $\alg A =\< A, \dots\>$.
For elements $a, b, a_j, b_j \in A$, the following are some shorthands
we will use to denote the congruences generated by these elements:
\[
\thetaab:= \Cg^{\alg{A}}(a, b) \qquad
\thetai:= \Cg^{\alg{A}}(a_i, b_i).
\]
Let $i \in \{0,1\}$.
By a \defin{local difference term operation for $(a,b,i)$}
we mean a ternary term operation $t$ satisfying the following conditions:
\begin{align}
\text{ if $i=0$, then } & a \comr{\thetaab} t(a,b,b); \label{eq:diff-triple1}\\
\text{ if $i=1$, then } & t(a,a,b) = b. \label{eq:diff-triple2}
\end{align}
If $t$ satisfies conditions~(\ref{eq:diff-triple1}) and~(\ref{eq:diff-triple2})
for all triples in some subset $S\subseteq A^2 \times \{0,1\}$, then we call $t$
a \defin{local difference term operation for $S$}.
Throughout the remainder of the paper, we will
write ``\ldto'' as a shorthand for
``local difference term operation.''

% Often we will take $S$ to be a finite list of elements of
% $A^2 \times \{0,1\}$, and $|S|$ will denote the length of
% the list. The definition above of a \ldto for
% a set has an obvious analog for a list.

\newcommand\dtrel{\ensuremath{\mathrel{\mathcal{D}}}}
\newcommand\dtr{\ensuremath{\mathcal{D}}}
A few more notational conventions will come in handy below.
Let $\Clo_3(\alg{A})$ denote the set of all ternary term operations of $\alg{A}$,
and let
$\dtr \subseteq \bigl(A^2\times \{0,1\}\bigr) \times \Clo_3(\alg{A})$
denote the relation that associates $(a, b, i)$
with the operations in $\Clo_3(\alg{A})$ that are \ldtos for
$(a, b, i)$.  That is,
$((a,b,i), t) \in \dtr$ iff conditions~(\ref{eq:diff-triple1})
and~(\ref{eq:diff-triple2}) are satisfied.
The relation $\dtr$ induces an obvious Galois connection
from subsets of $A^2\times \{0,1\}$ to subsets
of $\Clo_3(\alg{A})$.  Overloading the symbol $\dtr$,
we take
\begin{align*}
\dtr &\colon \sP(A^2 \times \{0,1\}) \to \sP(\Clo_3(\alg{A})) \text{ and }\\
\breve{\dtr} &\colon \sP(\Clo_3(\alg{A})) \to \sP(A^2\times \{0,1\})
\end{align*}
to be the maps defined as follows:
\begin{align*}
  \text{for }& \;  S \subseteq A^2 \times \{0,1\}\; \text{ and } \;T \subseteq \Clo_3(\alg{A}),
  \text{ let }\\
  \dtr S &= \{t \in \Clo_3(\alg{A}) \colon (s,t) \in \dtr \text{ for all } s \in S\}, \text{ and }\\
  \breve{\dtr} T &= \{s \in A^2 \times \{0,1\} \colon (s,t) \in \dtr \text{ for all } t \in T\}.
\end{align*}
In other words, $\dtr S$ is the set of \ldtos
for $S$, and $\breve{\dtr} T$ is the set of triples for which every $t\in T$
is a \ldto.  When the set $S$ is just a singleton, we write
$\dtr (a,b,\chi)$ instead of $\dtr \{(a,b,\chi)\}$.

Now, suppose that every pair
$(s_0, s_1)\in (\AsqBool)^2$ %$\AsqBool$
has a \ldto.
Then every subset $S\subseteq \AsqBool$
has a \ldto, as we now prove.

\begin{theorem} %[\protect{cf.~\cite[Theorem 2.2]{MR3239624}}]
  \label{thm:local-diff-terms}
  Let
  $\alg A$ be a finite idempotent algebra. %% . Define
  %% $\sS= A \times A \times \{0,1\}$
  If every pair
  $(s_0, s_1) \in (A^2 \times \{0,1\})^2$
  has a local difference term operation, then
  every subset $S \subseteq \AsqBool$
  has a local difference term operation.
\end{theorem}

\begin{proof}

The proof is by induction on the size of $S$.  In the base case,
$|S| = 2$, the claim holds by assumption. Fix $\ell\geq 2$ and assume that
every subset of $\AsqBool$ of size $k \leq \ell$ has a \ldto. Let
\[
S := \{(a_0, b_0, \chi_0), (a_1, b_1, \chi_1), \dots,
        (a_{\ell}, b_{\ell},\chi_{\ell})\} \subseteq \AsqBool,
\]
so $|S| = \ell+1$.  We prove $S$ has a \ldto.

Since $|S| \geq 3$, % and $\chi_i \in \{0,1\}$ for all $i$,
there exist indices $k\neq j$ such that $\chi_k = \chi_j$.
Assume without loss of generality that one such index is $j=\ell$,
and define the set
\[
S' := S \mysetminus \{(a_\ell, b_\ell, \chi_\ell)\}.
\]
Since $|S'| < |S|$, there exists $p \in \dtr S'$.
We split the remainder of the proof into two cases.

\medskip

%--------------------------------------

\newcommand{\thetal}{\ensuremath{\theta_{\ell}}}
\noindent \underline{Case $\chi_\ell = 0$}:
Without loss of generality, assume
\begin{equation*}
  \chi_0 = %% \chi_2 =
\cdots =\chi_{k-1} = 1 \quad \text{and} \quad
\chi_{k} = \cdots = \chi_{\ell} = 0.
\end{equation*}
If we define %% $T$ to be the set
\[S_0 := \{(a_0, b_0, 1), (a_1, b_1, 1),
\dots, (a_{k-1}, b_{k-1}, 1), (a_\ell, p(a_\ell, b_\ell, b_\ell), 0)\},\]
then $|S_0| < |S|$, so $S_0$ has a \ldto, $q \in \dtr S_0$.
We now show that
\[
d(x,y,z) := q(x, p(x,y,y), p(x,y,z))
\]
is a local difference term operation for $S$.

Since $\chi_\ell =0$, we must check that
$a_\ell \comr{\thetal} d(a_\ell,b_\ell,b_\ell)$.
If $\gamma := \Cg(a_\ell, p(a_\ell,b_\ell,b_\ell))$, then
\begin{equation}
    \label{eq:100000}
  d(a_\ell,b_\ell,b_\ell) =
  q(a_\ell, p(a_\ell,b_\ell,b_\ell), p(a_\ell,b_\ell,b_\ell))\comr{\gamma} a_\ell.
\end{equation}
The pair $(a_\ell, p(a_\ell,b_\ell,b_\ell))$ is equal to
$(p(a_\ell,a_\ell,a_\ell), p(a_\ell,b_\ell,b_\ell))$ and so
% belongs to $\theta_n:= \Cg^{\alg{A}}(a_n, b_n)$.
belongs to $\thetal$.
Therefore, $\gamma\leq \thetal$, so
$\com{\gamma} \leq \com{\thetal}$.
% by monotonicity of the commutator.
It follows from this and (\ref{eq:100000}) that
$a_\ell \comr{\thetal} d(a_\ell,b_\ell,b_\ell)$, as desired.

For indices $i < k$, we have $\chi_i =1$, so
$d(a_i,a_i,b_i) = b_i$ for such $i$. Indeed,
\[
  d(a_i,a_i,b_i) =
  q(a_i, p(a_i,a_i,a_i), p(a_i,a_i,b_i)) % \label{eq:200000}\\
  =q(a_i, a_i, b_i) % \label{eq:200001}\\
  =b_i. % \label{eq:200002}
\]
The first equation holds by definition of $d$, the second
because $p$ is an idempotent \ldto for
$S'$, and the third because $q \in \dtr S_0$.

The remaining triples in our original set $S$
have indices satisfying $k\leq j < \ell$ and $\chi_j = 0$.
Here, we have $a_j \comr{\thetaj} d(a_j,b_j,b_j)$. Indeed,
by definition,
\begin{equation}
  \label{eq:450000}
d(a_j,b_j,b_j) =q(a_j, p(a_j,b_j,b_j), p(a_j,b_j,b_j)),
\end{equation}
and, since $p \in \dtr S'$, we have
%% the pair $(p(a_j,b_j,b_j), a_j)$ belongs to $\com{\thetaj}$.
$a_j \comr{\thetaj} p(a_j,b_j,b_j)$,
so (\ref{eq:450000}) implies that
$a_j = q(a_j,a_j,a_j) \comr{\thetaj} d(a_j, b_j,b_j)$.

%% Finally, by idempotence of $q$ we have
%% $d(a_j,b_j,b_j)\comr{\thetaj} a_j$,
%% as desired.

\medskip
%--------------------------------------
\noindent \underline{Case $\chi_\ell = 1$}:
Without loss of generality, assume
\begin{equation*}
  \chi_0 =\cdots =\chi_{k-1} = 0
  \quad \text{and} \quad
\chi_{k} = \cdots = \chi_{\ell} = 1.
\end{equation*}
If
%% \begin{equation*}
$S_1 := \{(a_0, b_0, 0), (a_1, b_1 0), \dots, (a_{k-1}, b_{k-1}, 0),
        (p(a_\ell, a_\ell, b_\ell), b_\ell, 1)\}$,\\[3pt]
%% \end{equation*}
then $|S_1| < |S|$, so there exists $q \in \dtr S_1$.
We claim  that
\begin{equation*}
  d(x,y,z) := q(p(x,y,z), p(y,y,z), z)
\end{equation*}
is a \ldto for $S$. For $(a_\ell,b_\ell,\chi_\ell) \in S$ we have that
\[
d(a_\ell,a_\ell,b_\ell) = q(p(a_\ell,a_\ell,b_\ell), p(a_\ell,a_\ell,b_\ell), b_\ell) =b_\ell.
\]
% \end{equation*}
The last equality holds since $q \in \dtr S_1$.
%% $q$ is \ldto for $(p(a_n, a_n, b_n), b_n, 1)$.

If $i < k$, then $\chi_i =0$. For these indices we must prove
that $a_i$ is congruent to $d(a_i,b_i,b_i)$ modulo $\com{\thetai}$.
Again, starting from the definition of $d$ and using idempotence of $p$, we have
\begin{equation}
  \label{eq:40000}
  d(a_i,b_i,b_i) =
  q(p(a_i,b_i,b_i), p(b_i,b_i,b_i), b_i)=
  q(p(a_i,b_i,b_i), b_i, b_i).
\end{equation}
Next, since $p \in \dtr S'$,
\begin{equation}
  \label{eq:50000}
  q(p(a_i,b_i,b_i), b_i, b_i)
 \comr{\thetai}
 q(a_i, b_i, b_i).
\end{equation}
Since $q \in \dtr S_1$, we have
$q(a_i, b_i, b_i) \comr{\thetai} a_i$, so
(\ref{eq:40000}) and (\ref{eq:50000}) imply
$d(a_i,b_i,b_i) \comr{\thetai} a_i$, as desired.

The remaining elements of $S$
have indices satisfying $k\leq j < \ell$ and $\chi_j = 1$.
For these we want $d(a_j,a_j,b_j) = b_j$.
Since $p \in \dtr S'$, we have
$p(a_j,a_j,b_j) = b_j$, and this plus idempotence of $q$ yields
\begin{equation*}
 d(a_j,a_j,b_j) =  q(p(a_j,a_j,b_j), p(a_j,a_j,b_j), b_j)=  q(b_j, b_j, b_j) =b_j,
\end{equation*}
as desired.
\end{proof}

\begin{corollary}
  \label{cor:loc-diff-term}
  Let $\alg A$ be a finite idempotent algebra and suppose that
  every pair $(s, s') \in (\AsqBool)^2$ has a local difference term operation.
  Then $\dtr (\AsqBool) \neq \emptyset$,
  so $\alg{A}$ has a difference term operation.
\end{corollary}
\begin{proof}
  Letting $S := \AsqBool$ in Theorem~\ref{thm:local-diff-terms} establishes
  the existence of a \ldto $d$ for $S$.  That is, $d \in \dtr S$.
  It follows that $d$ is a difference
  term operation for $A$. Indeed, for all $a, b \in A$, we have that
  $a \comr{\thetaab} d(a,b,b)$, since $d \in \dtr S \subseteq \dtr (a,b,0)$,
  %$d$ is \ldto for $(a,b,0)$,
  and $d(a,a,b) = b$, since $d\in \dtr S \subseteq \dtr (a,b,1)$.
  %$d$ is \ldto for $(a,b,1)$.
\end{proof}
%
% \begin{corollary}
%
%   A finite idempotent algebra $\alg A $ has a difference term operation if and
%   only if each pair in $(A^2 \times \{0,1\})^2$
%   has a local difference term.
% \end{corollary}
% \begin{proof}
%   A difference term operation for $\alg A $ is clearly a \ldto for every pair in
%   $(A^2 \times \{0,1\})^2$, so one direction of the corollary is obvious.
%   For the converse, suppose
%   each pair in $(A^2 \times \{0,1\})^2$ has a \ldto.
%   Then, by Theorem~\ref{thm:local-diff-terms},
%   there is a single \ldto for the whole set $A^2 \times \{0,1\}$,
%   and this is a difference term operation for $\alg A $.
% \end{proof}

% \draftsecskip

\subsection{Test for existence of a difference term operation}
\label{sec:algor-1}
Here is a practical consequence of Theorem~\ref{thm:local-diff-terms}.
\begin{corollary}
  \label{cor:algor-1}
  There is a polynomial-time algorithm that takes as input
  any finite idempotent algebra $\alg A $ and decides if
  %% the variety $\bbV(\alg A )$ that it generates
  $\alg A $ has a difference term operation.
\end{corollary}
%\mav{Add Cor showing how to test if V(A) has a dt. (See Keith's proof.)}

\begin{proof}
  %% and let  $\sV = \bbV(\alg A )$.
  We describe an efficient algorithm for deciding,
  given a finite idempotent algebra $\alg A $,
  whether every pair in $(A^2 \times \{0,1\})^2$
  has a \ldto.  By Corollary~\ref{cor:loc-diff-term}, this will prove we
  can decide in polynomial-time whether $\alg A $ has a difference term operation.
  %% We will then complete the
  %% proof by explaining why $\alg A $ has a difference term operation iff the variety
  %% it generates has a difference term.

  Fix a pair
  $((a,b,i), (a',b',i'))$ in $(A^2 \times \{0,1\})^2$. If $i = i' = 0$,
  then the first projection is a \ldto. If $i = i' = 1$,
  then the third projection is a \ldto. The two remaining cases
  occur when $i\neq i'$. Without loss of generality, assume $i = 0$ and $i'=1$,
  so the given pair of triples is of the form $((a,b,0), (a',b',1))$.
  By definition, $t \in \dtr \{(a,b,0), (a',b',1)\}$ iff
  \[
  a\comr{\thetaab} t^{\alg A }(a,b,b) \; \text{ and } \;
  t^{\alg A }(a',a',b') = b'.
  \]
  We can rewrite this condition more compactly by noting that
  \[t^{\alg{A} \times \alg{A} }((a,a'), (b,a'), (b,b')) =
  (t^{\alg A }(a,b,b),t^{\alg A }(a',a',b')),\]
  and that
  $t \in \dtr \{(a,b,0), (a',b',1)\}$ if and only if
  \[
  t^{\alg A \times \alg A }((a,a'), (b,a'), (b,b'))\in a/\delta \times \{b'\},
  \]
  where $\delta = \com{\thetaab}$ and $a/\delta$ denotes the
  $\delta$-class containing $a$.
  It follows that $\dtr \{(a,b,0), (a',b',1)\} \neq \emptyset$
  %has a \ldto
  iff the subuniverse of $\alg A \times \alg A $ generated by
  $\{(a,a'), (b,a'), (b,b')\}$ intersects nontrivially with the subuniverse
  $a/\delta \times \{b'\}$.

  Thus, we take as input a finite idempotent algebra $\alg A $ and,
  for each element $((a,a'), (b,a'), (b,b'))$ of $(A\times A)^3$,
  \begin{enumerate}
    \item compute $\thetaab$,
    \item compute $\delta = \com{\thetaab}$,
    \item compute $\bS = \Sg^{\alg A \times \alg A }\{(a,a'), (b,a'), (b,b')\}$,
    \item \label{item:a3} test whether $S \cap (a/\delta \times \{b'\})$ is empty.
  \end{enumerate}
  If ever we find an empty intersection in step (\ref{item:a3}), then
  $\alg A $ has no difference term operation.
  Otherwise the algorithm halts without witnessing an empty
  intersection, in which case $\alg A $ has a difference term operation.

  Finally, we analyze the time-complexity of the procedure just described,
  using the same notation and complexity bounds as those appearing in
  Section~\ref{sec:algorithm-its-time}.  Recall, $n = |A|$, and
  $m=\|\alg A\| = \sum_{i=0}^r k_i n^i$, where $k_i$ is the number of basic
  operations of arity~$i$, and $r$ is the largest arity of the basic
  operations of $\alg{A}$. The following assertions are consequences of (the
  proof of) Proposition~\ref{speedprop}: $\thetaab$,
  $\delta$, and $S$ are computable in time $O(rm)$, $O(rm^2 + n^5)$,
  and $O(rm^2)$, respectively;
  $\thetaab$ and $\delta$ are computed for each pair $(a,b) \in A^2$;
  $S$ is computed for each triple of the form
  $((a,a'),(b,a'),(b,b'))\in (A\times A)^3$, and there are $n^4$ such triples.
  Thus, the computational complexity of the above procedure is
  $O(rm^2n^4 + n^7)$.
\end{proof}

\subsection{Test for existence of a difference term}
\label{sec:diffterm-test}
The main result of this subsection is Proposition~\ref{prop:alt-poly-alg}, which provides another avenue for constructing a polynomial-time algorithm to decide if a finite idempotent algebra generates a variety that has a difference term (cf.~Theorem~\ref{thm:time}).
One step of the proof requires a minor technical lemma, so we dispense with that obligation first.
\begin{lemma} \label{lem:dt-quotients}
  Let $\alg A$ be a finite idempotent algebra.
  If $\alg A$ has a difference term operation then so does every quotient of $\alg A$.
\end{lemma}

\begin{proof}
  Let $t^{\alg A}$ be a difference term operation for $\alg A$.  Fix $\gamma \in \Con \alg A$.
We show $t^{\alg A/\gamma}$ is a difference term operation for $\alg A/\gamma$.  Let $a$, $b \in A$ and let $\theta:= \Cg^{\alg A}(a,b)$ and $\theta_\gamma := \Cg^{\alg A/\gamma}(a/\gamma, b/\gamma)$.
Obviously, $t^{\alg A}(a,a,b) = b$  implies
$t^{\alg A/\gamma}(a/\gamma, a/\gamma, b/\gamma) = b/\gamma$.

We need to show that $(a/\gamma, t^{\alg A}(a,b,b)/\gamma) \in \com{\theta_\gamma}$.  It will suffice to show that if $\beta \in \con(\alg A/\gamma)$  is such that $C(\theta_\gamma, \theta_\gamma; \beta)$ then $(a/\gamma, t^{\alg A}(a,b,b)/\gamma) \in \beta$.  But such a congruence $\beta$ is of the form $\delta/\gamma$ for some $\delta \in \con(\alg A)$ with $\delta \ge \gamma$. Also, the congruence $\theta_\gamma$ is equal to $(\gamma \join \theta)/\gamma$ and so the condition $C(\theta_\gamma, \theta_\gamma; \beta)$ can be restated as
\[
C((\gamma \join \theta)/\gamma, (\gamma \join \theta)/\gamma; \delta/\gamma).
\]
By Proposition 3.4 (5) of~\cite{HM:1988}, this is equivalent to
\[
C((\gamma \join \theta), (\gamma \join \theta); \delta).
\]
From this condition, we can use the monotonicity of the centrality condition (see Proposition 3.4 (1) of~\cite{HM:1988}) to deduce that $C(\theta, \theta; \delta)$. Since $t$ is a difference term operation for $\alg A$ then $(a, t^{\alg A}(a,b,b)) \in \delta$ and so $(a/\gamma, t^{\alg A}(a,b,b)/\gamma) \in  \delta/\gamma = \beta$ as required.
\end{proof}

\begin{proposition}\label{prop:alt-poly-alg}
  Let $\alg A$ be a finite idempotent algebra.  Then $\bbV (\alg A)$ has a difference term if and only if each 3-generated subalgebra of $\alg A^2$ has a difference term operation.
\end{proposition}

\begin{proof}
  Of course if $\bbV (\alg A)$ has a difference term then each 3-generated subalgebra of
  $\alg A^2$ has a difference term operation. For the converse, we refer to the proof of
  Theorem~\ref{thm:KearnesThm}~\cite[Theorem 3.8]{MR1358491}.  One part of this theorem
  establishes that if an algebra $\alg B$ is in a variety that has a difference term,
  then all type~\atyp minimal sets of $\alg B$ have empty tails.  A careful reading of the
  proof of this fact shows that a weaker hypothesis will suffice, namely that all quotients
  of $\alg B$ have difference term operations.  By Lemma~\ref{lem:dt-quotients}, this is
  equivalent to just $\alg B$ having a difference term operation.
  %% since this property is preserved under taking quotients.
  %% \mav{This is true, I think.  Should it be mentioned somewhere, as part of a general discussion of difference term operations?}

  To complete the proof of this proposition, we use Corollary~\ref{cor:diffterm}. Suppose that each 3-generated subalgebra of $\alg A^2$ has a difference term operation.  Then, in particular, each 2-generated subalgebra of $\alg A$ does.  This rules out that $\sansH \sansS (\alg{A})$ contains an algebra that is term equivalent to the 2-element set.  It also follows, from the previous paragraph, that no 3-generated subalgebra of $\alg A^2$ has a prime quotient of type~\atyp whose minimal sets have nonempty tails.
\end{proof}

\begin{corollary}
  There is a polynomial-time algorithm to decide if a finite idempotent algebra $\alg A$ generates a variety that has a difference term.
\end{corollary}

\begin{proof}
  By the previous proposition, it suffices to check whether each 3-generated subalgebra of
  $\alg A^2$ has a difference term operation. This can be decided by applying the
  algorithm from Corollary~\ref{cor:algor-1} at most $\binom{n^2}{3}$ ($\approx n^6$) times,
  so the total running time of this decision procedure is $O(rm^2n^{10} + n^{13})$ (cf.~$O(rn^4m^4 + n^{14}),$
  the complexity of the difference term existence test of Theorem~\ref{thm:time}).
\end{proof}

% \mav{If this is correct, should we say anything about the complexity of this algorithm?  Is it more efficient than the one from earlier in the paper?}
\section{Efficiently computing difference term tables}
\label{sec:comp-diff-term}
In this section we present a polynomial-time algorithm that takes
a finite idempotent algebra $\alg{A}$ and constructs the Cayley table of
a difference term operation for $\alg{A}$, if such a term exists.

The method consists of three subroutines.
Algorithm~\ref{alg:ld-2} finds Cayley tables of \ldtos for sets of size 2,
and Algorithm~\ref{alg:new-2-1} calls Algorithm~\ref{alg:ld-2} repeatedly to find
tables of \ldtos for larger and larger sets up to size $n^2+1$.
Finally, Algorithm~\ref{alg:new-2-3}
calls these subroutines in order to produce \ldtos for larger subsets of
$A^2 \times \{0,1\}$.

Some new notation will be helpful here.  Suppose that $ u \in \Sg^{\alg{A}}(X)$.
Then there is a term $t$ of some arity $k$ that, when applied to a certain tuple of
generators $(x_1, \dots, x_k)$, produces $u$.
In this situation, we may ask for the operation table (or ``Cayley table'')
for $t^{\alg{A}}$, which is an $A^k$-dimensional array whose $(a_1, \dots, a_k)$-entry
is the value $t^{\alg{A}}(a_1, \dots, a_k)$.
If $\sanst$ is such a table, we will denote the
entry at position $(a_1, \dots, a_k)$ of the table by $\sanst[a_1, \dots, a_k]$.
Alternatively, if we wish to emphasize the means by which
we arrived at the term that the table represents, we may use
$\sanst_{x_1,\dots, x_k \mid u}$ to denote the table.

\newcommand{\tripzu}{\ensuremath{(a_0, b_0, 0), (a_1, b_1, 1)}}

\newcommand{\triptik}{\ensuremath{(a,a'),(b,a'),(b,b')}}
\newcommand{\pairtik}{\ensuremath{(a,b, 0), (a', b', 1)}}
\newcommand{\abb}{\ensuremath{(a, b, b)}}
\newcommand{\aabtik}{\ensuremath{(a', a', b')}}

\newcommand{\cabxbtik}{\ensuremath{a/\com{\theta_{ab}} \times \{b'\}}}
\subsection{Base step} %Algorithm~\ref{alg:ld-2}: implementation and complexity}
\label{sec:cc-ld-2}
The first step of our method finds a ternary operation on
$\alg{A} \times \alg{A}$ that maps the element
$(\triptik)$ into the set
$\cabxbtik$. In other words, the first step finds a \ldto for
$\{\pairtik\}$.
When such an operation is found, its Cayley table---a
$3$-dimensional array $\sanst$ satisfying
$\sanst[a,b, b] \mathrel{\delta} a$ and  $\sanst[a',a',b'] = b'$---is
returned.

\RestyleAlgo{boxruled}%\LinesNumbered
\begin{algorithm}%[ht]

  \KwIn{$S = \{\pairtik\}$}
  \KwOut{Cayley table of an operation in $\dtr S$}

  compute $\delta = \com{ \thetaab }$ and form $C= (a/\delta) \times \{b'\}$\;
  compute $S=\Sg^{\alg{A}\times \alg{A}} (\triptik)$\;
  \ForAll{$(u,v)  \in S$} {
    compute the table $\sanst_{\triptik\mid (u,v)}$\;
    \If{$(u,v) \in C$} {
      \Return $\sanst_{\triptik \mid (u,v)}$\;
    }
  \Return false (i.e., there is no difference term operation)
  }

  \caption{Generate the Cayley table of a \ldto for $\{\pairtik \}$}
  \label{alg:ld-2}
  % in the $((a_0,a_1),(b_0,a_1),(b_0,b_1))$ position

\end{algorithm}
Note that the subalgebra $S$ in Algorithm~\ref{alg:ld-2}
need not be computed in its entirety before
the condition inside the {\bf forall} loop is tested.  Naturally, we test
$(u,v) \in C$ as soon as the new element $(u,v) \in S$ is generated.

Let us now consider the computational complexity of
Algorithm~\ref{alg:ld-2}.
Recall the notation introduced in Section~\ref{sec:algorithm-its-time};
\begin{align*}
  n &= |A|, \quad m=\|\alg A\| = \sum_{i=0}^r k_i n^i,\\
k_i&= \text{ the number of basic operations of arity~$i$},\\
r &= \text{ the largest arity of the basic operations of $\alg{A}$.}
\end{align*}
Also, $\sanst_{\triptik |(u,v)}$ denotes the Cayley table of a
term operation that generates $(u,v)$ from the set $\{\triptik \}$.

Algorithm~\ref{alg:ld-2} can be implemented as follows:
\begin{enumerate}
  \item compute $\thetaab$, in time $O(rm)$;
  \item compute $C= a/\com{ \thetaab } \times \{b'\}$,
  in time $O(rm^2 + n^5)$;
  \item generate $S=\Sg^{\alg{A}\times \alg{A}} \{\triptik\}$,
    in time $O(r m^2)$;\\
    for each newly generated $(u,v) \in S$,
  \begin{itemize}
    \item construct and store the table
      $\sanst_{\triptik |(u,v)}$;
    \item if $(u,v) \in C$, then return $\sanst_{\triptik |(u,v)}$;
  \end{itemize}
\end{enumerate}

Each element $(u,v)\in S$ is the result of applying some (say, $k$-ary)
basic operation $f$ to previously generated pairs $(u_1, v_1)$, $\dots$, $(u_k,v_k)$
from $S$, and the operation tables generating these pairs were
already stored (the first bullet of Step 3).  Thus, to compute the table
for the operation that produced $(u,v) = f((u_1,v_1),\dots, (u_k,v_k))$
we simply compose $f$ with previously stored operation tables.
Since all tables represent ternary operations, the time-complexity of this composition
is $|A|^3$-steps multiplied by $k$ reads per step; that is,
$kn^3 \leq rn^3$.
All told, the time-complexity of Algorithm~\ref{alg:ld-2} is %$O(rm + rm^2 + n^5 + r^2m^2n^3) =
$O(r^2m^2n^3 + n^5)$.

\subsection{Inductive stages}
The method is based on the proof of Theorem~\ref{thm:local-diff-terms} and consists of $n^2$
stages, each of which makes $n^2$ calls to Algorithm~\ref{alg:ld-2}.

\renewcommand{\l}{\ensuremath{\ell}}
Let $\l := n^2 -1$ and let $\{(a_0, b_0), (a_1, b_1), (a_2, b_2), \dots, (a_{\ell}, b_{\ell})\}$
be an enumeration of the set $A^2$.
For all $1\leq k \leq n^2$, define
\begin{align*}
  Z_k &:= \{(a_0, b_0,0),(a_1, b_1,0), \dots, (a_{k-1}, b_{k-1},0)\},\\
  U_k &:= \{(a_0, b_0,1),(a_1, b_1,1), \dots, (a_{k-1}, b_{k-1},1)\}.
\end{align*}
So $Z_{n^2} \cup U_{n^2} = A^2 \times \{0,1\}$.
The first stage of our procedure computes the Cayley table of a \ldto for
$Z_1 \cup U_{n^2}:= \{(a_0, b_0, 0)\} \cup A^2 \times \{1\}$.
The second stage does the same for
$Z_2 \cup U_{n^2} := \{(a_1, b_1, 0) \} \cup Z_1 \cup U_{n^2}$.
This continues for $n^2$ stages, after which we obtain a Cayley table of a \ldto for
$Z_{n^2} \cup U_{n^2} = A^2 \times \{0,1\}$.
There are $n^2$ stages, each stage consisting of $n^2$ steps (described below),
and each step requiring a single call to Algorithm~\ref{alg:ld-2}.
Thus, the procedure makes $n^4$ calls to Algorithm~\ref{alg:ld-2}, and
the total running-time is on the order of $r^2m^2n^7 + n^9$.

The following is a list of steps taken in Stage 1 in order to compute a \ldto for
$\{(a_0, b_0,0),(a_0, b_0,1),(a_1, b_1,1), \dots, (a_{\ell}, b_{\ell},1)\}$:
\begin{enumerate}[1.]
  \item compute the table $\sanst_1$ of a \ldt op for $\{(a_0, b_0, 0), (a_0, b_0, 1)\}$;\\[-8pt]
  \item compute the table $\sanss_1$ of a \ldt op for $\{(a_0, b_0, 0), (\sanst_1[a_1, a_1, b_1], b_1, 1)\}$; \\
  form the table $\sanst_2[x,y,z] = \sanss_1\bigl[\sanst_1[x,y,z], \sanst_1[y,y,z], z\bigr]$
  ($\forall x,y,z$);
  \\[-8pt]
  \item compute the table $\sanss_2$ of a \ldt op for $\{(a_0, b_0, 0), (\sanst_2[a_2, a_2, b_2], b_2, 1)\}$; \\
    form the table $\sanst_3[x,y,z] = \sanss_2\bigl[\sanst_2[x,y,z], \sanst_2[y,y,z], z\bigr]$  ($\forall x,y,z$);\\
  $\vdots$
  \item[~\hskip-5mm$n^2$\!.] compute a table $\sanss_{\l}$ of a \ldt op for $\{(a_0, b_0, 0),  (\sanst_{\l}[a_{\l}, a_{\l}, b_{\l}], b_{\l}, 1)\}$; \\
  form the table $\sanst_{n^2}[x,y,z] = \sanss_{\l}\bigl[\sanst_{\l}[x,y,z], \sanst_{\l}[y,y,z], z\bigr]$
  ($\forall x,y,z$).
\end{enumerate}
Let $\sansd_1 := \sanst_{n^2}$ denote the final result of Stage 1. It is not hard to check that
$\sansd_1$ is the table of a \ldto for $Z_1 \cup U_{n^2}$. (See the proof of Theorem~\ref{thm:local-diff-terms}.)

Stage 2 is very similar to Stage 1; however, we first produce the table, $\sansd'_2$, of a \ldto for
$\{(a_1, \sansd_1[a_1, b_1, b_1], 0)\} \cup U_{n^2}$, and then form the table
$\sansd_2[x,y,z] := \sansd'_2\bigl[x, \sansd_1[x,y,y], \sansd_1[x,y,z]\bigr]$ ($\forall x,y,z \in A$),
which will be the table of a \ldto for $Z_2 \cup U_{n^2}$.
We continue in this way for $n^2$ stages, of $n^2$ steps each, until we reach
our goal: $\sansd:=\sansd_{n^2}$, the Cayley table of
a \ldto for $Z_{n^2} \cup U_{n^2} = A^2 \times \{0,1\}$.

Without further ado, here is a precise description of an algorithm that carries out any one
of the $n^2$ stages; the argument $(a,b)$ determines which stage is executed.
\RestyleAlgo{boxruled}
%\LinesNumbered
\begin{algorithm}
  \KwIn{A pair $(a,b) \in A^2$}
  \KwOut{The Cayley table of a \ldt op for $\{(a,b, 0)\} \cup U_{n^2}$.}

  Use Alg.~\ref{alg:ld-2} to compute the table $\sanst_1$ of a \ldto for $\{(a,b, 0), (a_0, b_0, 1)\}$\;

  \ForAll {$1\leq i < {n^2}$} {
    Use Alg.~\ref{alg:ld-2} to compute the table $\sanss_i$ of a \ldto for $\{(a,b, 0), (\sanst_i[a_i, a_i, b_i], b_i, 1)\}$\\[6pt]
    Form the table $\sanst_{i+1}$, defined as follows: $\forall x,y,z$,\\
    $\quad \sanst_{i+1}[x,y,z] = \sanss_i\bigl[\sanst_i[x,y,z], \sanst_i[y,y,z], z\bigr]$\;
  }

  \Return $\sanst_{n^2}$
  \caption{Return the Cayley table of a \ldto for $\{(a,b,0)\}\cup U_{n^2}$\label{alg:new-2-1}}
\end{algorithm}

\begin{lemma}
  \label{lem:19}
  The output of Algorithm~\ref{alg:new-2-1} is the Cayley table of a \ldto for
  $\{(a,b,0)\} \cup U_{n^2}$. %:=  \{(a, b,0),(a_0, b_0,1), (a_1, b_1,1), \dots, (a_{\l}, b_{\l}, 1)\}$.
\end{lemma}

\begin{proof}
  This follows from the proof of Theorem~\ref{thm:local-diff-terms}.
\end{proof}

\begin{algorithm}
\KwOut{$\sansd_{n^2}$, the Cayley table of a \ldt op for $A^2 \times \{0,1\}$}

  Use Alg.~\ref{alg:new-2-1} to compute the table $\sansd_1$ of a \ldt op for
  $Z_1 \cup U_{n^2}$\; %(Alg.~\ref{alg:new-2-1});

  \ForAll {$1\leq k <{n^2}$} {

    Use Alg.~\ref{alg:new-2-1} compute the table $\sansd'_{k+1}$ of a \ldto for $\{(a_k, \sansd_{k}[a_k, b_k, b_k], 0)\} \cup U_{n^2}$;\\[4pt]
    %(Alg.~\ref{alg:new-2-1});\\[4pt]
    Form the table $\sansd_{k+1}$, defined as follows: $\forall x,y,z$,\\
    $\quad \sansd_{k+1}[x,y,z] := \sansd'_{k+1}\bigl[x, \sansd_{k}[x,y,y], \sansd_{k}[x,y,z]\bigr]$\;
  }
  \Return $\sansd_{n^2}$
  \caption{Return the Cayley table of a difference term operation for $\alg{A}$\label{alg:new-2-3}}
\end{algorithm}

\begin{proposition}
  The output of Algorithm~\ref{alg:new-2-3} is the Cayley table of a \ldto for $A^2 \times \{0,1\}$, hence a difference term operation for $\alg{A}$.
\end{proposition}

\begin{proof}
This follows from Lemma~\ref{lem:19} and the proof of Theorem~\ref{thm:local-diff-terms}.
\end{proof}

\section*{Acknowledgments}
The first and second authors were supported by the National
Science Foundation under Grant Numbers 1500218 and 1500235; the third author was supported by a grant from the Natural Sciences and Engineering Research Council of Canada.

\bibliographystyle{ws-ijac}
\bibliography{refs}

  %%%%%%%%%%%%%%%%%%%%%%%%%%%%%%%%%%%%%%%%%%%%%%%%%%%%%%%%%%%%%%%%%%%%%%%%%%%%%%%%%%%%
  %%%%%%%%%%%%%%%%%%%%%%%%%%%%  END OF DOCUMENT %%%%%%%%%%%%%%%%%%%%%%%%%%%%%%%%%%%%%%
  %%%%%%%%%%%%%%%%%%%%%%%%%%%%%%%%%%%%%%%%%%%%%%%%%%%%%%%%%%%%%%%%%%%%%%%%%%%%%%%%%%%%
  %%%%%%%%%%%%%%%%%%%%%%%%%%%%%%%%%%%%%%%%%%%%%%%%%%%%%%%%%%%%%%%%%%%%%%%%%%%%%%%%%%%%
  %%%%%%%%%%%%%%%%%%%%%%%%%%%%  END OF DOCUMENT %%%%%%%%%%%%%%%%%%%%%%%%%%%%%%%%%%%%%%
  %%%%%%%%%%%%%%%%%%%%%%%%%%%%%%%%%%%%%%%%%%%%%%%%%%%%%%%%%%%%%%%%%%%%%%%%%%%%%%%%%%%%
  %%%%%%%%%%%%%%%%%%%%%%%%%%%%%%%%%%%%%%%%%%%%%%%%%%%%%%%%%%%%%%%%%%%%%%%%%%%%%%%%%%%%
  %%%%%%%%%%%%%%%%%%%%%%%%%%%%  END OF DOCUMENT %%%%%%%%%%%%%%%%%%%%%%%%%%%%%%%%%%%%%%
  %%%%%%%%%%%%%%%%%%%%%%%%%%%%%%%%%%%%%%%%%%%%%%%%%%%%%%%%%%%%%%%%%%%%%%%%%%%%%%%%%%%%
\end{document}